\documentclass{amsart}

\RequirePackage{fix-cm}
\usepackage{graphicx}
\usepackage{amsmath, amsopn,amstext,amscd,amsfonts,amssymb}
\usepackage{dsfont}
\usepackage{comment}
\usepackage[active]{srcltx}
\usepackage{graphicx, epsfig, subfig}

\usepackage{textcomp}

\usepackage{algorithm}

\makeatletter
\def\BState{\State\hskip-\ALG@thistlm}
\makeatother

\def\downbar#1{
\setbox10=\hbox{$#1$}
            \dimen10=\ht10 \advance\dimen10 by 2.5pt
            \ifdim \dimen10<15pt 
               \advance\dimen10 by -0.5pt
               \dimen11=\dimen10
               \advance\dimen10 by 2.5pt
               \lower \dimen11
            \else \lower \ht10 \fi
            \hbox {\hskip 1.5pt \vrule height \dimen10 depth \dp10}}
\def\upbar#1{
\setbox10=\hbox{$#1$}
            \dimen10=\ht10 \advance\dimen10 by \dp10 \advance\dimen10 by 2.5pt
            \ifdim \dimen10<15pt 
                \advance\dimen10 by 2pt \fi
            \raise 2.5pt \hbox {\hskip -1.5pt \vrule height \dimen10}}

\usepackage{multicol}
\usepackage{colortbl}
\usepackage{exscale}
\usepackage{amssymb,latexsym,amsthm,amsfonts,color,fancyhdr,mathrsfs}
\usepackage{lscape}
\usepackage{rotating}
\usepackage{caption}

\newtheorem{definition}{\bf Definition}[section]
\newtheorem{theorem}{\bf Theorem}[section]
\newtheorem{proposition}{\bf Proposition}[section]

\newtheorem{corollary}{\bf Corollary}[section]
\newtheorem{remark}{\bf Remark}[section]

\numberwithin{equation}{section}
\begin{document}
\title[Remarks on classical orthogonal polynomials]{Remarks on Askey-Wilson polynomials and Meixner polynomials of the second kind}

\author{K. Castillo}
\address{University of Coimbra, CMUC, Dep. Mathematics, 3001-501 Coimbra, Portugal}
\email{kenier@mat.uc.pt}
\author{D. Mbouna}
\address{University of Coimbra, CMUC, Dep. Mathematics, 3001-501 Coimbra, Portugal}
\email{dmbouna@mat.uc.pt}
\author{J. Petronilho}
\address{University of Coimbra, CMUC, Dep. Mathematics, 3001-501 Coimbra, Portugal}
\email{josep@mat.uc.pt}

\subjclass[2010]{42C05, 33C45}
\date{\today}
\keywords{Askey-Wilson polynomials, Meixner polynomials of the second kind}

\begin{abstract}
The purpose of this note is twofold: firstly to characterize all the sequences of orthogonal polynomials $(P_n)_{n\geq 0}$ such that
$$
\frac{\triangle}{{\bf \triangle} x(s-1/2)}P_{n+1}(x(s-1/2))=c_n(\triangle +2\,\mathrm{I})P_n(x(s-1/2)),
$$
where $\,\mathrm{I}$ is the identity operator, $x$ defines a class of lattices with, generally, nonuniform step-size, and $\triangle f(s)=f(s+1)-f(s)$; and secondly to present, in a friendly way, a method to deal with these kind of problems. 
\end{abstract}
\maketitle

\section{Introduction}\label{introduction}
In 1972, Al-Salam and Chihara proved (see \cite{Al-Salam-1972}) that $(P_n)_{n\geq 0}$ is a $\mathrm{D}$-classical orthogonal polynomial sequence (OPS), namely Hermite, Laguerre, Bessel or Jacobi families, if and only if 
\begin{align}
(az^2+bz+c)\mathrm{D}P_n(z)=(a_nz+b_n)P_n(z)+c_nP_{n-1}(z)\quad (c_n\neq 0)\;, \label{very-classical}
\end{align}
where $\mathrm{D}=d/dz$. If we replace $\mathrm{D}$, in \eqref{very-classical}, by the Hahn operator, depending on two parameters $q$ and $\omega$, we come naturally to the concept of $(q, \omega)$-classical OPS. For $\omega=0$, Datta and Griffin (see \cite{DG2006}) stated that the only solutions of the corresponding equation  \eqref{very-classical} are Al-Salam Carlitz I, little and big $q$-Laguerre, little and big $q$-Jacobi, and $q$-Bessel polynomials. It is worth pointing out that without assuming $\omega=0$, two additional families appear (see \cite{RKDP2020} and references therein). In the same way, we can replace $\mathrm{D}$, in \eqref{very-classical}, by the Askey-Wilson operator. The problem of characterizing such OPS was posed by Ismail (see \cite[Conjecture 24.7.8]{I2005}). The case $a=b=0$ and $c=1$ was considered by Al-Salam (see \cite{A-1995}). Recently, we addressed this problem in its full generality (see \cite{KDPconj}), which leads to a characterization of continuous $q$-Jacobi and some special cases of the Al-Salam-Chihara polynomials. However, the “How” is sometimes more important than the “What”, and the methods presented in \cite{KDPconj} allow us to easily address these issues. Since \cite{KDPconj} is a much more technical work, in this note we will show the reader how to use the ideas developed therein. In order to do this, let us consider the following difference equation
\begin{align}\label{open-problem}
&(az^2+bz+c)\frac{\triangle}{\triangle x(s-1/2)} P_n(x(s-1/2))\\
\nonumber&\qquad=(\triangle +2\,\mathrm{I})(a_nP_{n+1}+b_nP_n+c_nP_{n-1})(x(s-1/2)),
\end{align}
where $\,\mathrm{I}$ is the identity operator, $x$ defines a class of lattices (or grids) with, generally, nonuniform step-size, $\triangle f(s)=f(s+1)-f(s)$, and $\nabla f(s)=\triangle f(s-1)$. Our objective is to present an analog to the Al-Salam theorem \cite{A-1995}; in other words, to characterize for $a=b=0$ and $c=1$ the OPS that satisfy \eqref{open-problem}. The general case is currently being studied, but the calculations involved are too heavy and once again, as in \cite{KDPconj}, the reader would get lost among them without understanding the simplicity of the proposed method. 

The structure of the paper is as follows. Section 2 presents some basic facts of the algebraic theory of OPS on lattices and the Askey-Wilson polynomials and Meixner polynomials of the second kind. Section \ref{preliminaries} contains some preliminary results. In Section \ref{S-main} our main results are stated and proved.  

\section{Background}
Let $\mathcal{P}$ be the vector space of all polynomials with complex coefficients
and let $\mathcal{P}^*$ be its algebraic dual. A simple set in $\mathcal{P}$ is a sequence $(P_n)_{n\geq0}$ such that $\mathrm{deg}(P_n)=n$ for each $n$. A simple set $(P_n)_{n\geq0}$ is called an OPS with respect to ${\bf u}\in\mathcal{P}^*$ if 
$$
\langle{\bf u},P_nP_m\rangle=\kappa_n\delta_{n,m}\quad(m=0,1,\ldots;\;\kappa_n\in\mathbb{C}\setminus\{0\}),
$$
where $\langle{\bf u},f\rangle$ is the action of ${\bf u}$ on $f\in\mathcal{P}$. In this case, we say that ${\bf u}$ is  regular. The left multiplication of a functional ${\bf u}$ by a polynomial $\phi$ is defined by
$$
\left\langle \phi {\bf u}, f  \right\rangle =\left\langle {\bf u},\phi f  \right\rangle \quad (f\in \mathcal{P}).
$$
Consequently, if $(P_n)_{n\geq0}$ is a monic OPS with respect to ${\bf u}\in\mathcal{P}^*$, then the corresponding dual basis is explicitly given by 
\begin{align}\label{expression-an}
{\bf a}_n =\left\langle {\bf u} , P_n ^2 \right\rangle ^{-1} P_n{\bf u}.
\end{align}
Any functional ${\bf u} \in \mathcal{P}^*$ (when $\mathcal{P}$ is endowed with an appropriate strict inductive limit topology, see \cite{M1991}) can be written in the sense of the weak topology in $\mathcal{P}^*$ as 
\begin{align*}
{\bf u} = \sum_{n=0} ^{\infty} \left\langle {\bf u}, P_n \right\rangle {\bf a}_n.
\end{align*}
It is known (see \cite{C1978}) that a monic OPS, $(P_n)_{n\geq 0}$, is characterized by the following three-term recurrence relation (TTRR):
\begin{align}\label{TTRR_relation}
P_{-1} (z)=0, \quad P_{n+1} (z) =(z-B_n)P_n (z)-C_n P_{n-1} (z) \quad (C_n \neq 0),
\end{align}
and, therefore,
\begin{align}\label{TTRR_coefficients}
B_n = \frac{\left\langle {\bf u} , xP_n ^2 \right\rangle}{\left\langle {\bf u} , P_n ^2 \right\rangle},\quad C_{n+1}  = \frac{\left\langle {\bf u} , P_{n+1} ^2 \right\rangle}{\left\langle {\bf u} , P_n ^2 \right\rangle}.
\end{align}

In our framework, a lattice $x$ is a mapping given by (see \cite{ARS1995})
\begin{equation}
\label{xs-def}
x(s)=\left\{
\begin{array}{lcl}
\mathfrak{c}_1 q^{-s} +\mathfrak{c}_2 q^s +\mathfrak{c}_3,&  q\neq1\\ [7pt]
\mathfrak{c}_4 s^2 + \mathfrak{c}_5 s +\mathfrak{c}_6, &  q =1,
\end{array}
\right.
\end{equation}
where $q>0$ and $\mathfrak{c}_j$ ($1\leq j\leq6$) are complex numbers such that $(\mathfrak{c}_1,\mathfrak{c}_2)\neq(0,0)$ if $q\neq1$.
Note that
$x\big(s+\frac12\big)+x\big(s-\frac12\big)=2\alpha x(s)+2\beta,$
where
\begin{equation}\label{alpha-beta}
\alpha=\frac{q^{1/2}+q^{-1/2}}{2},\quad
\beta=\left\{
\begin{array}{lcl}
(1-\alpha)\mathfrak{c}_3, &  q\neq1,\\ [7pt]
\mathfrak{c}_4/4, &  q =1.
\end{array}
\right.
\end{equation}
Moreover, 
\begin{align*}
\frac{x(s+n)+x(s)}{2}&=\alpha_n x_n (s) +\beta_n , ~
x(s+n)-x(s)=\gamma_n \nabla x_{n+1} (s),
\end{align*}
where $x_{\mu}(s)=x(s+\mu/2)$ and $(\alpha_n)_{n\geq0}$, $(\beta_n)_{n\geq0}$, and $(\gamma_n)_{n\geq0}$ are given by
\begin{align*}
2\alpha_n&=q^{n/2} +q^{-n/2},\\[7pt]
\beta_n&= \left\{
\begin{array}{lcl}
\displaystyle (1-\alpha_n)\mathfrak{c}_3, & q\neq1\\[7pt]
\beta\,n^2, & q=1,
\end{array}
\right.  \quad \gamma_n = \left\{
\begin{array}{lcl}
\displaystyle\frac{q^{n/2}-q^{-n/2}}{q^{1/2}-q^{-1/2}}, & q\neq1 \\ [7pt]
n, & q=1.
\end{array}
\right. 
\end{align*}
One may easily check that
\begin{align}\label{form-4}
\alpha +\alpha_n\gamma_n=\alpha_{n-1}\gamma_{n+1}.  
\end{align}
Define two operators $\mathrm{D}_x$ and $\mathrm{S}_x$ on $\mathcal{P}$ by 
\begin{align*}
\mathrm{D}_x f(x(s))=\frac{\triangle}{{ \triangle} x(s-1/2)}f(x(s-1/2)),\quad \mathrm{S}_x f(x(s))= \frac{1}{2}(\triangle+2\,\mathrm{I})f(x(s-1/2)),
\end{align*}
These operators  induce two elements on $\mathcal{P}^*$, say $\mathbf{D}_x$ and $\mathbf{S}_x$, via the following definition (see \cite{FK-NM2011}): 
\begin{align*}
\langle \mathbf{D}_x{\bf u},f\rangle=-\langle {\bf u},\mathrm{D}_x f\rangle,\quad \langle\mathbf{S}_x{\bf u},f\rangle=\langle {\bf u},\mathrm{S}_x f\rangle.
\end{align*}

The monic Askey-Wilson polynomial, $(Q_n(\cdot; a_1, a_2, a_3, a_4 | q))_{n\geq 0}$, satisfy \eqref{TTRR_relation} (see \cite[(14.1.5)]{KLS2010}) with
\begin{align*}
B_n &= a_1+\frac{1}{a_1}-\frac{(1-a_1a_2q^n)(1-a_1a_3q^n)(1-a_1a_4q^n)(1-a_1a_2a_3a_4q^{n-1})}{a_1(1-a_1a_2a_3a_4q^{2n-1})(1-a_1a_2a_3a_4q^{2n})}\\[7pt]
&\quad-\frac{a_1(1-q^n)(1-a_2a_3q^{n-1})(1-a_2a_4q^{n-1})(1-a_3a_4q^{n-1})}{(1-a_1a_2a_3a_4q^{2n-1})(1-a_1a_2a_3a_4q^{2n-2})},\\[7pt]
C_{n+1}&=(1-q^{n+1})(1-a_1a_2a_3a_4q^{n-1}) \\[7pt]
&\quad\times \frac{(1-a_1a_2q^n)(1-a_1a_3q^n)(1-a_1a_4q^n)(1-a_2a_3q^n)(1-a_2a_4q^n)(1-a_3a_4q^n)}{4(1-a_1a_2a_3a_4q^{2n-1})(1-a_1a_2a_3a_4q^{2n})^2 (1-a_1a_2a_3a_4q^{2n+1})}
\end{align*}
and subject to the following restrictions (see \cite{KDP2021}):
$$
\begin{array}l
(1-a_1a_2a_3a_4q^n)(1-a_1a_2q^n)(1-a_1a_3q^n) \\[7pt]
\qquad\quad\times(1-a_1a_4q^n)(1-a_2a_3q^n)(1-a_2a_4q^n)(1-a_3a_4q^n) \neq 0.
\end{array}
$$
The monic Meixner polynomials of the second kind, $(M_n(\cdot;b_1,b_2))_{n\geq 0}$, are defined by (see \cite[p.179, (3.17)]{C1978})
\begin{align}
zM_n(z;b_1,b_2)&=M_{n+1}(z;b_1,b_2)-b_1(2n+b_2) M_n(z;b_1,b_2) \label{Meixner2nd} \\ 
&\quad +(b_1^2+1)n(n+b_2-1)M_{n-1}(z;b_1,b_2),\quad M_{-1}(z;b_1,b_2)=0, \nonumber
\end{align}
where $b_1$ and $b_2$ are parameters so that $b_1^2\neq-1$ and $b_2\neq0,-1,-2,\ldots$.

\section{Preliminary results}\label{preliminaries}
 Let $f,g\in\mathcal{P}$ and ${\bf u}\in\mathcal{P}^*$. Then the following properties hold (see e.g. \cite{KDP2021} and references therein):
\begin{align}
\mathrm{D}_x \big(fg\big)&= \big(\mathrm{D}_x f\big)\big(\mathrm{S}_x g\big)+\big(\mathrm{S}_x f\big)\big(\mathrm{D}_x g\big), \label{def-Dx-fg} \\[7pt]
\mathrm{S}_x \big( fg\big)&=\big(\mathrm{D}_x f\big) \big(\mathrm{D}_x g\big)\texttt{U}_2  +\big(\mathrm{S}_x f\big) \big(\mathrm{S}_x g\big), \label{def-Sx-fg} \\[7pt]
f{\bf D}_x {\bf u}&={\bf D}_x\left(S_xf~{\bf u}  \right)-{\bf S}_x\left(D_xf~{\bf u}  \right), \label{def-fD_x-u}\\[7pt]
\alpha \mathbf{D}_x ^n \mathbf{S}_x {\bf u}&= \alpha_{n+1} \mathbf{S}_x \mathbf{D}_x^n {\bf u}
+\gamma_n \texttt{U}_1\mathbf{D}_x^{n+1}{\bf u}, \label{DxnSx-u} 
\end{align}
where
\begin{align*}
\texttt{U}_1 (z)&=\left\{
\begin{array}{lcl}
(\alpha^2-1)\big(z-\mathfrak{c}_3\big), & q\neq1\\[7pt]
2\beta, &q=1,
\end{array}
\right. \\[7pt]
\texttt{U}_2(z)&=\left\{
\begin{array}{lcl}
(\alpha^2-1)\big((z-\mathfrak{c}_3)^2-4\mathfrak{c}_1\mathfrak{c}_2\big), & q\neq1\\[7pt]
4\beta (z-\mathfrak{c}_6)+\mathfrak{c}_5^2/4, & q=1.
\end{array}
\right.
\end{align*}
It is known that if $\;x(s)=\mathfrak{c}_1q^{-s}+\mathfrak{c}_2q^s+\mathfrak{c}_3$, then
\begin{align}
\mathrm{D}_x z^n =\gamma_n z^{n-1}+u_nz^{n-2}+\cdots,\quad \mathrm{S}_x z^n =\alpha_n z^n+\widehat{u}_nz^{n-1}+\cdots, \label{Dx-xnSx-xn}
\end{align}
where $u_n =\big(n\gamma_{n-1}-(n-1)\gamma_n\big)\mathfrak{c}_3$ and $\widehat{u}_n =n(\alpha_{n-1}-\alpha_n)\mathfrak{c}_3$. For quadratic lattices we present the following proposition.
\begin{proposition}
Consider the quadratic lattice $x(s)=4\beta s^2 +\mathfrak{c}_5s+\mathfrak{c}_6$. Then 
\begin{align}
\mathrm{D}_x z^n =n z^{n-1}+v_nz^{n-2}+\cdots,\quad 
\mathrm{S}_x z^n =z^n+\widehat{v}_nz^{n-1}+\cdots\; \quad (n=1,2,\ldots),\label{DxSx-xn-quadratic}
\end{align}
where $v_n=\beta n(n-1)(2n-1)/3$ and $\widehat{v}_n=\beta n(2n-1)$. 
\end{proposition}
\begin{proof}
For $n=1$, we have $\mathrm{D}_x z =1$ and $\mathrm{S}_x z =z+\beta$, and so $v_1=0$ and $\widehat{v}_1 =\beta$. Then (\ref{DxSx-xn-quadratic}) is true for $n=1$. Now suppose that (\ref{DxSx-xn-quadratic}) is true for all integers less than or equal to a fixed $n$. Using this together with (\ref{def-Dx-fg}) and (\ref{def-Sx-fg}), we have
\begin{align*}
\mathrm{D}_x z^{n+1} &=\mathrm{D}_x (zz^n)=\mathrm{D}_x z^n ~ \mathrm{S}_x z + \mathrm{S}_x z^n ~\mathrm{D}_x z=(z+\beta)\mathrm{D}_x z^n + \mathrm{S}_x z^n \\[7pt]
&=(n+1)z^n +(v_n +\widehat{v}_n  +\beta n)z^{n-1}+\cdots= (n+1)z^n +v_{n+1}z^{n-1}+\cdots.
\end{align*}
In a similar way we also have
\begin{align*}
\mathrm{S}_x z^{n+1} &=\mathrm{S}_x (zz^n)=\texttt{U}_2 (z)\mathrm{D}_x z~\mathrm{D}_x z^n  + \mathrm{S}_x z^n ~\mathrm{S}_x z=\texttt{U}_2 (z)\mathrm{D}_x z^n + ( z+\beta)\mathrm{S}_x z^n\\[7pt]
&=z^{n+1} +(4\beta n+\beta +\widehat{v}_n)z^n +\cdots=z^{n+1}+\widehat{v}_{n+1}z^n +\cdots,
\end{align*}
and the result follows.
\end{proof}

\begin{definition}\cite{FK-NM2011, KDP2021}
Let $x$ be a lattice given by \eqref{xs-def}. 
${\bf u}\in\mathcal{P}^*$ is called $x$-classical if it is regular and there exist polynomials $\phi$ and $\psi$ with $deg( \phi)\leq 2$ and $deg( \psi)=1$ such that
\begin{equation}\label{NUL-Pearson}
\mathbf{D}_x(\phi{\bf u})=\mathbf{S}_x(\psi{\bf u}).
\end{equation}
An OPS with respect to a $x$-classical functional is called a $x$-classical OPS.
\end{definition}

\begin{theorem}\label{main-Thm1}\cite{KDP2021}
Let $(P_n)_{n\geq 0}$ be a monic OPS with respect to ${\bf u} \in \mathcal{P}^*$. 
Suppose that ${\bf u}$ satisfies \eqref{NUL-Pearson} where $\phi(z)=az^2+bz+c$ and $\psi(z)=dz+e$, with $d\neq0$.
Then $(P_n)_{n\geq 0}$ satisfies \eqref{TTRR_relation} with
\begin{align}
B_n  =\mathfrak{c}_3+ \frac{\gamma_n e_{n-1}}{d_{2n-2}}
-\frac{\gamma_{n+1}e_n}{d_{2n}},\quad
C_{n+1}  =-\frac{\gamma_{n+1}d_{n-1}}{d_{2n-1}d_{2n+1}}\phi^{[n]}\left(\mathfrak{c}_3 -\frac{e_{n}}{d_{2n}}\right),\label{Bn-Cn-Dx}
\end{align}
 where $d_n=a\gamma_n+d\alpha_n$, $e_n=\phi'(\mathfrak{c}_3)\gamma_n+\psi(\mathfrak{c}_3)\alpha_n$, and \begin{align*}
\phi^{[n]}(z)&=\big(d(\alpha^2-1)\gamma_{2n}+a\alpha_{2n}\big)
\big((z-\mathfrak{c}_3)^2-2\mathfrak{c}_1\mathfrak{c}_2\big)\\[7pt]
&\quad +\big(\phi'(\mathfrak{c}_3)\alpha_n+\psi(\mathfrak{c}_3)(\alpha^2-1)\gamma_n\big)(z-\mathfrak{c}_3)
+ \phi(\mathfrak{c}_3)+2a\mathfrak{c}_1\mathfrak{c}_2,
\end{align*}
if $x(s)=\mathfrak{c}_1 q^{-s} +\mathfrak{c}_2 q^s +\mathfrak{c}_3$, or else
\begin{align}\label{Bn-Cn-quadratic}
B_n = \frac{ne_{n-1}}{d_{2n-2}} -\frac{(n+1)e_n}{d_{2n}} -2\beta n(n-1),\; \;  
C_{n+1} =-\frac{(n+1)d_{n-1}}{d_{2n-1}d_{2n+1}}\phi ^{[n]}\left(-\beta n^2 -\frac{e_n}{d_{2n}}  \right),
\end{align} 
where $d_n =an+d$, $e_n=bn+e+2\beta dn^2$, and 
$$\phi ^{[n]}(z)=az^2 +(b+6\beta nd_n)z+ \phi(\beta n^2)+2\beta n\psi(\beta n^2)-\frac{n}{4}\left( 16\beta \mathfrak{c}_6 -\mathfrak{c}_5 ^2\right)d_n
$$ 
otherwise.
\end{theorem}

\begin{remark}
Under the hypothesis of Theorem \ref{main-Thm1}, it was also proved in \cite{KDP2021} that the conditions $d_n\neq0$ and 
\begin{align}\label{Reg1Rmk}
\phi^{[n]}\left(\mathfrak{c}_3 -\frac{e_{n}}{d_{2n}}\right)\neq0
\end{align}
if $x(s)=\mathfrak{c}_1 q^{-s} +\mathfrak{c}_2 q^s +\mathfrak{c}_3$, or else $d_n\neq0$ and
\begin{align}\label{Reg2Rmk}
\phi ^{[n]}\left(-\beta n^2 -\frac{e_n}{d_{2n}}\right)\neq0
\end{align}
otherwise, hold. Moreover, these are necessary and sufficient conditions for the regularity of a nonzero functional ${\bf u}\in\mathcal{P}^*$ fulfilling \eqref{NUL-Pearson}.
\end{remark}

Theorem \ref{main-Thm1} will play a crucial role along this work. We denote by $P_n ^{[k]}$ $(k=0,1,\ldots)$ the monic polynomial of degree $n$ defined by
\begin{align*}
P_n ^{[k]} (z)=\frac{\mathrm{D}_x ^k P_{n+k} (z)}{ \prod_{j=1} ^k \gamma_{n+j}} =\frac{\gamma_{n} !}{\gamma_{n+k} !} \mathrm{D}_x ^k P_{n+k} (z). 
\end{align*}
Here it is understood that $\mathrm{D}_x ^0 f=f $, empty product equals one, and $\gamma_0 !=1$, $\gamma_{n+1}!=\gamma_1\cdots \gamma_n \gamma_{n+1}$. If $({\bf a}^{[k]} _n)_{n\geq 0}$ is the dual basis associated to the sequence $(P_n ^{[k]})_{n\geq 0}$, we leave it to the reader to verify that
\begin{align}
{\bf D}_x ^k {\bf a}^{[k]} _n=(-1)^k \frac{\gamma_{n+k}!}{\gamma_n ! }{\bf a}_{n+k}\quad (k=0,1,\ldots). \label{basis-Dx-derivatives}
\end{align}

\begin{proposition}
Let $(P_n)_{n\geq 0}$ be a monic OPS with respect to ${\bf u}\in \mathcal{P}^*$. Assume that 
\begin{align}
\mathrm{D}_xP_{n+1}(z)= k_n \mathrm{S}_xP_n(z)\quad (n=0,1,\ldots;\;k_n\in \mathbb{C}).\label{equation-to-solve}
\end{align}
Then $k_n=\alpha_n^{-1} \gamma_{n+1}$ and
\begin{align}
{\bf D}_x ( ( \gamma_{n+1} \mbox{\rm $\texttt{U}_1$}  P_{n+1}+\alpha_nC_{n+1}P_n){\bf u})=-\alpha \gamma_{n+1} {\bf S}_x (P_{n+1}{\bf u} ). \label{x-classical-with-n}
\end{align}
\end{proposition}  

\begin{proof} 
The expression for $k_n$ is obtained by identifying the leading coefficient on each member of \eqref{equation-to-solve} using \eqref{Dx-xnSx-xn}.
Let $({\bf a}_n)_{n\geq 0}$ and $({\bf a} ^{[1]} _n)_{n\geq 0}$ be the dual basis associated to the sequences $(P_n)_{n\geq 0}$ and $(P ^{[1]} _n)_{n\geq 0}$, respectively. We claim that 
\begin{align}\label{kn-and-Sxan1}
{\bf S}_x{\bf a}^{[1]} _n=\alpha_n {\bf a}_n.
\end{align}
Indeed, by \eqref{equation-to-solve}, we have 
$$\left\langle {\bf S}_x {\bf a}^{[1]} _n,P_j  \right\rangle =\left\langle  {\bf a}^{[1]} _n,\mathrm{S}_xP_j  \right\rangle = k_j ^{-1}\gamma_{j+1} \left\langle {\bf a}^{[1]} _n,P^{[1]} _{j}  \right\rangle=\alpha_j \delta_{n,j}, \quad (j=0,1,\dots).$$ 
Hence
$$ {\bf S}_x {\bf a}_n ^{[1]} =\sum_{j=0} ^{+\infty} \left\langle {\bf S}_x {\bf a}^{[1]} _n,P_j  \right\rangle {\bf a}_j=\alpha_n {\bf a}_n.$$
We now apply the operator ${\bf D}_x$ to \eqref{kn-and-Sxan1}, using \eqref{DxnSx-u} (for $n=1$ and replacing ${\bf u}$ by ${\bf a}_n ^{[1]}$) and \eqref{basis-Dx-derivatives} (for $k=1$) to obtain
\begin{align}
-\alpha\alpha_n {\bf D}_x {\bf a}_n =(2\alpha ^2-1)\gamma_{n+1}{\bf S}_x {\bf a}_{n+1} +\gamma_{n+1}\texttt{U}_1 {\bf D}_x {\bf a}_{n+1}.\label{almost-done}
\end{align}
Moreover, from \eqref{def-fD_x-u} with $f$ and ${\bf u}$ replaced by $\texttt{U}_1$ and ${\bf a}_{n+1}$, respectively, and using the fact that $\mathrm{D}_x\texttt{U}_1=\alpha^2-1$ and $\mathrm{S}_x\texttt{U}_1=\alpha \texttt{U}_1$,
we obtain the following equation
\begin{align}\label{intermediate-r}
\texttt{U}_1 {\bf D}_x {\bf a}_{n+1}=\alpha {\bf D}_x (\texttt{U}_1 {\bf a}_{n+1})-(\alpha^2-1){\bf S}_x {\bf a}_{n+1}.
\end{align}
Putting \eqref{intermediate-r} inside \eqref{almost-done} yields
\begin{align*}
{\bf D}_x (\alpha_n {\bf a}_{n}+\gamma_{n+1}\texttt{U}_1{\bf a}_{n+1})=-\alpha \gamma_{n+1}{\bf S}_x {\bf a}_{n+1}.
\end{align*}
Finally, \eqref{x-classical-with-n} follows from the above equation taking into account \eqref{expression-an}, \eqref{TTRR_relation}, and \eqref{TTRR_coefficients}.
\end{proof}

The next results follows immediately by taking $n=0$ in \eqref{x-classical-with-n}.

\begin{corollary}\label{coro-sol}
Let $(P_n)_{n\geq 0}$ be a monic OPS with respect to ${\bf u}\in \mathcal{P}^*$. Assume that \eqref{equation-to-solve} holds. Then ${\bf u}$ is $x$-classical. Moreover, 
\begin{align}
{\bf D}_x(\phi {\bf u})={\bf S}_x (\psi {\bf u}), \label{true-x-GP-equation}
\end{align}
where
\begin{align*}
\psi(z)=z-B_0,\quad 
\phi(z)=\left\{
\begin{array}{lcl}
-(\alpha -\alpha^{-1})(z-\mathfrak{c}_3)(z-B_0)-\alpha^{-1}C_1, &  q\neq1,\\ [7pt]
-2\beta(z-B_0)-C_1, &  q =1.
\end{array}
\right.
\end{align*}
\end{corollary}

\begin{remark}\label{Rmk1}
According to Corollary \ref{coro-sol}, any OPS $(P_n)_{n\geq0}$ satisfying \eqref{equation-to-solve} is $x$-classical and, therefore, Theorem \ref{main-Thm1} can be applied to determine the recurrence coefficients, $B_n$ and $C_n$, appearing in \eqref{TTRR_relation} in term of $B_0$ and $C_1$, which may be regarded as the only possible free parameters. We will see in the next section, for each case where the lattice is fixed, that we need to take into account some initial conditions which will allow us to obtain completely all possible solutions. 
\end{remark}

\section{Main results}\label{S-main}

 We are now in position to state our main results.

\begin{theorem}\label{ThmMain1}
Consider the lattice $x(s)=\mathfrak{c}_1q^{-s} +\mathfrak{c}_2 q^s +\mathfrak{c}_3$ with $\mathfrak{c}_1\mathfrak{c}_2\neq 0$. 
Then, up to an affine transformation of the variable, the only monic OPS, $(P_n)_{n\geq 0}$, satisfying 
\begin{align}\label{equat-solving}
\mathrm{D}_xP_{n+1}(z)=\alpha_n^{-1}\gamma_{n+1}\mathrm{S}_xP_n(z),
\end{align}
are those of the Askey-Wilson polynomials
$$
P_n(z)=2^n(\mathfrak{c}_1\mathfrak{c}_2)^{n/2}Q_n\left(\frac{z-\mathfrak{c}_3}{2\sqrt{\mathfrak{c}_1\mathfrak{c}_2}};a,-a,iq^{-1/2}/a,-iq^{-1/2}/a\Big|q\right), 
$$
with $a\notin \left\lbrace \pm q^{(n-1)/2}, \pm iq^{-n/2}\,|\,n=0,1,\ldots\right\rbrace$.
\end{theorem}

\begin{proof}
Let ${\bf u}\in \mathcal{P}^*$ be the regular functional with respect to which $(P_n)_{n\geq 0}$ is an OPS.  We claim that $B_n$, in \eqref{TTRR_relation}, is given by
\begin{align}\label{Bn-sol-q-quadratic}
B_n =\mathfrak{c}_3\quad (n=0,1,\ldots).
\end{align}  
Indeed, recall that $P_n(z)=z^n+f_n z^{n-1}+\cdots$, where $f_0=0$ and $f_n=-\sum_{j=0} ^{n-1}B_j$  $(n=1,2,\dots)$.
Using \eqref{Dx-xnSx-xn}, identifying the second coefficient of higher degree in both sides of \eqref{equat-solving}, yields
$$
\alpha_n (u_{n+1}+\gamma_nf_{n+1})=\gamma_{n+1}(\widehat{u}_n+\alpha_{n-1}f_n).
$$
Hence, by \eqref{Dx-xnSx-xn} and \eqref{form-4}, we can rewrite this equation as
$$
\frac{\alpha_n}{\gamma_{n+1}}f_{n+1}=\frac{\alpha_{n-1}}{\gamma_n}f_n +\frac{n\alpha-\alpha_n\gamma_n}{\gamma_{n}\gamma_{n+1}}\mathfrak{c}_3.
$$
By the telescopic sum method, we have
\begin{align*}
f_n=-\frac{\gamma_n}{\alpha_{n-1}}\left(B_0 +\mathfrak{c}_3\sum_{j=1} ^{n-1}\frac{\alpha_j\gamma_j- \alpha j}{\gamma_{j}\gamma_{j+1}}    \right)=-\frac{\gamma_n}{\alpha_{n-1}}(B_0-\mathfrak{c}_3)-n\mathfrak{c}_3.
\end{align*}
Therefore using \eqref{form-4} we obtain
\begin{align}\label{first-Bn-q}
B_n =f_n-f_{n+1}=\mathfrak{c}_3+\frac{\alpha}{\alpha_{n-1}\alpha_n}(B_0-\mathfrak{c}_3).
\end{align}
By Corollary \ref{coro-sol}, ${\bf u}$ satisfies \eqref{true-x-GP-equation} where $\phi(z)=-(\alpha-\alpha^{-1})(z-B_0)(z-\mathfrak{c}_3)-\alpha^{-1}C_1$ and $\psi(z)=z-B_0$. From \eqref{Bn-Cn-Dx}, we obtain
\begin{align}\label{second-Bn-q}
B_n=\mathfrak{c}_3 +(1+q)(B_0-\mathfrak{c}_3)q^{n-2}\frac{(q-1)(1-q^{2n-2})+(1+q)q^{n-1}}{(1+q^{2n-3})(1+q^{2n-1})}.
\end{align} 
If $0<q<1$, combining \eqref{first-Bn-q} and \eqref{second-Bn-q}, we obtain
\begin{align*}
2(1+q^{-1})(B_0-\mathfrak{c}_3)=\lim_{n\rightarrow +\infty} q^{-n}(B_n-\mathfrak{c}_3)=(1-q^{-2})(B_0-\mathfrak{c}_3),
\end{align*}
and \eqref{Bn-sol-q-quadratic} follows. Similarly, if $1<q<\infty$, then 
$$
2(1+q)(B_0-\mathfrak{c}_3)=\lim_{n\rightarrow +\infty}q^{n}(B_n-\mathfrak{c}_3)=(1-q^2)(B_0-\mathfrak{c}_3),
$$
and \eqref{Bn-sol-q-quadratic} also follows.
Therefore, the above expressions for $\phi$ and $\psi$ reduce to
$$\phi(z)=-(\alpha-\alpha^{-1})(z-\mathfrak{c}_3)^2-\alpha^{-1}C_1\;,\quad
\psi(z)=z-\mathfrak{c}_3.$$ 
It follows that $C_1$ can be regarded as the only free parameter (cf. Remark \ref{Rmk1}). 
From \eqref{Bn-Cn-Dx} we obtain $d_n= \alpha^{-1}\alpha_{n-1}$ and
\begin{align}
\phi ^{[n]}\Big(\mathfrak{c}_3-\frac{e_n}{d_{2n}}\Big)&=\frac{1}{2\alpha}\Big(\mathfrak{c}_1\mathfrak{c}_2(1-q^{-1})(1-q^{n})(1+q^{-n+1})-2C_1  \Big). \label{Eqphinc3}
\end{align}
Choose a parameter $r$ as a solution of the quadratic equation
$$
(q-1)\mathfrak{c}_1\mathfrak{c}_2 Z^2 +2(C_1 +2(\alpha^2-1)\mathfrak{c}_1\mathfrak{c}_2)Z-(1-q^{-1})\mathfrak{c}_1\mathfrak{c}_2=0;
$$
that is
$$
r= \frac{C_1+2(\alpha^2-1)\mathfrak{c}_1\mathfrak{c}_2}{(1-q)\mathfrak{c}_1\mathfrak{c}_2} \pm \sqrt{q^{-1}+\Big( \frac{C_1+2(\alpha^2-1)\mathfrak{c}_1\mathfrak{c}_2}{(1-q)\mathfrak{c}_1\mathfrak{c}_2} \Big)^2}.
$$ 
Instead of $C_1$ we may consider $r$ as the free parameter and to express $C_1$ in terms of $r$ as follows:
$$C_1=\frac12(1-q^{-1})(1+r^{-1})(1-rq)\mathfrak{c}_1\mathfrak{c}_2.$$ 
Therefore \eqref{Eqphinc3} can be rewriten as
\begin{align*}
\phi ^{[n]}\Big(\mathfrak{c}_3-\frac{e_n}{d_{2n}}\Big)&=\mathfrak{c}_1\mathfrak{c}_2\frac{1-q}{2\alpha}(1+rq^{n})(1-r^{-1}q^{n-1})q^{-n}. 
\end{align*}
It follows from the regularity conditions \eqref{Reg1Rmk} that the free parameter $r$ should satisfy the condition $r\notin \left\lbrace q^{n-1}, -q^{-n}\,|\,n=0,1,\ldots\right\rbrace$.
Moreover, \eqref{Bn-Cn-Dx} yields
\begin{align}
C_{n+1}=\mathfrak{c}_1\mathfrak{c}_2 \frac{(1+q^{n-2})(1-q^{n+1})(1+rq^n)(1-r^{-1}q^{n-1})}{(1+q^{2n-2})(1+q^{2n})}.
\end{align}
Thus 
$$P_n(x)=2^n(\mathfrak{c}_1\mathfrak{c}_2)^{n/2}Q_n\left(\frac{z-\mathfrak{c}_3}{2\sqrt{\mathfrak{c}_1\mathfrak{c}_2}};\sqrt{r},-\sqrt{r},i/\sqrt{rq},-i/\sqrt{rq} \,\Big|\,q\right),$$ 
and setting $r=a^2$ the theorem follows.
\end{proof}

\begin{theorem}\label{main-thm-quadratic}
Consider the lattice  $x(s)=4\beta s^2 +\mathfrak{c}_5 s +\mathfrak{c}_6$ with $(\beta,\mathfrak{c}_5)\neq (0,0)$.
Then there exist monic OPS, $(P_n)_{n\geq 0}$, satisfying 
\begin{align}\label{equat-solving2}
\mathrm{D}_xP_{n+1}(z)=(n+1)\mathrm{S}_xP_n(z)
\end{align}
if and only if $\beta=0$. In this case, up to an affine transformation of the variable, these polynomials are those of Meixner of the second kind
$$
P_n(z)=\left(\frac{i\mathfrak{c}_5}{2}\right)^nM_n\left(\frac{2i(B_0-z)}{\mathfrak{c}_5};0,-\frac{4C_1}{\mathfrak{c}_5 ^2}\right),
$$
with $B_0,C_1\in\mathbb{C}$ and  $4C_1/\mathfrak{c}_5^2\not\in\mathbb{N}$.
\end{theorem}

\begin{proof}
Let ${\bf u}\in \mathcal{P}^*$ be the regular functional with respect to which $(P_n)_{n\geq 0}$ is an OPS. Suppose that $(P_n)_{n\geq 0}$ satisfies \eqref{equat-solving2}. Then $B_n$, in \eqref{TTRR_relation}, is given by 
\begin{align}
B_n=B_0 -2\beta n(n-1)\;. \label{Bn-quadratic-1}
\end{align}
Indeed, as in the proof of Theorem \ref{ThmMain1}, from $P_n(z)=z^n +f_nz^{n-1}+\cdots$ and \eqref{equat-solving2}, and using \eqref{DxSx-xn-quadratic}, we obtain 
$$v_{n+1} +nf_{n+1}=(n+1)\widehat{v}_n +(n+1)f_n  \quad \quad (n=1,2,\ldots)\;.$$ 
This can be rewritten as
$$\frac{f_{n+1}}{n+1} =\frac{f_n}{n} +\frac{(n+1)\widehat{v}_n-v_{n+1}}{n(n+1)}.$$
Therefore, by the telescopic sum method, we obtain 
$$ f_n=-nB_0 +n\sum_{l=1} ^{n-1} \frac{(l+1)\widehat{v}_l-v_{l+1}}{l(l+1)}=-nB_0 +\frac{2}{3}\beta n(n-1)(n-2).$$
The expression of $B_n$, in \eqref{Bn-quadratic-1}, follows from the above equation and $B_n=f_n-f_{n+1}$. On the other hand, using Corollary \ref{coro-sol}, we see that ${\bf u}$ satisfies \eqref{true-x-GP-equation} with $\psi(z)=z-B_0$ and $\phi(z)=-2\beta(z-B_0)-C_1$. 
Hence \eqref{Bn-Cn-quadratic} yields $$B_n=B_0-8\beta n(n-1).$$ 
This agrees with the expression of $B_n$ given in \eqref{Bn-quadratic-1} if and only if $\beta=0$. Therefore, 
we compute $d_n=1$ and $\phi^{[n]}(z)=-C_1+n\mathfrak{c}_5 ^2/4$, and
using again \eqref{Bn-Cn-quadratic} we obtain 
\begin{align*}
B_n=B_0\;,\quad C_{n+1}=-\frac{1}{4}\mathfrak{c}_5^2 (n+1)\left(n-\frac{4C_1}{\mathfrak{c}_5 ^2}\right),
\end{align*}
$B_0$ and $C_1$ being free parameters subject to the regularity conditions  $\mathfrak{c}_5 ^2n-4C_1\neq0$ for each $n=0,1,\ldots$.
Finally, the theorem follows from \eqref{Meixner2nd}.
\end{proof}

\section*{Acknowledgements }
This work is supported by the Centre for Mathematics of the University of Coimbra-UID/MAT/00324/2019, funded by the Portuguese Government through FCT/MEC and co-funded by the European Regional Development Fund through the Partnership Agreement PT2020. 

{

\end{document}